\documentclass[amstex,12pt,reqno]{amsart}
\usepackage{amsmath, amssymb, amsthm}
\usepackage[dvipdfmx]{graphicx}

\newtheorem{theorem}{Theorem}[section]
\newtheorem{lemma}[theorem]{Lemma}
\newtheorem{proposition}[theorem]{Proposition}

\theoremstyle{definition}

\newtheorem*{remark}{Remark}

\title[Beurling--Ahlfors extension by heat kernel]
{Beurling--Ahlfors extension by heat kernel, ${\rm A}_\infty$-weights for VMO, and vanishing Carleson measures}

\author[H. Wei]{Huaying Wei} 
\address{Department of Mathematics and Statistics, Jiangsu Normal University \endgraf Xuzhou 221116, PR China} 
\email{hywei@jsnu.edu.cn} 

\author[K. Matsuzaki]{Katsuhiko Matsuzaki}
\address{Department of Mathematics, School of Education, Waseda University \endgraf
Shinjuku, Tokyo 169-8050, Japan}
\email{matsuzak@waseda.jp}

\thanks{Research supported by the National Natural Science Foundation of China (Grant No. 11501259)
and Japan Society for the Promotion of Science (KAKENHI 18H01125).}

\begin{document}

\maketitle

\begin{abstract}
We investigate a variant of the Beurling--Ahlfors extension 
of quasisymmetric homeo\-morphisms of the real line that is given by the convolution of the heat kernel, and prove that
the complex dilatation of such a quasiconformal extension of a strongly symmetric homeomorphism (i.e. its derivative is
an ${\rm A}_\infty$-weight whose logarithm is in VMO) induces a vanishing Carleson measure on the upper half-plane.
\end{abstract}

\section{Introduction}
Beurling and Ahlfors \cite{BA} characterized the boundary value of a quasiconformal homeo\-morphism of
the upper half-plane $\mathbb U$ onto itself as a quasisymmetric homeomorphism $f$ of the real line $\mathbb R$.
Here, an increasing homeomorphism $f:\mathbb R \to \mathbb R$ is {\it quasisymmetric} if there is a constant $\rho >1$ such that
$|f(2I)| \leq \rho |f(I)|$ for any bounded interval $I \subset \mathbb R$, where $|\cdot|$ is the Lebesgue measure and
$2I$ denotes the interval of the same center as $I$
with $|2I|=2|I|$.
They proved that any quasisymmetric homeomorphism of $\mathbb R$ extends continuously to 
a quasiconformal homeomorphism $F:\mathbb U \to \mathbb U$ in a certain explicit way. This is called the
{\it Beurling--Ahlfors extension}.

Let $\phi(x)=\frac{1}{2} 1_{[-1,1]}(x)$ and $\psi(x)=\frac{r}{2} 1_{[-1,0]}(x)+\frac{-r}{2} 1_{[0,1]}(x)$
for some $r>0$, where $1_E$ denotes the characteristic function of $E \subset \mathbb R$. For any function $\varphi(x)$ 
on $\mathbb R$ and for $t>0$, we set $\varphi_t(x)=\frac{1}{t} \varphi(\frac{x}{t})$.
Then, for a quasisymmetric homeomorphism $f$, the Beurling--Ahlfors extension
$F(x,t)=(U(x,t),V(x,t))$ for $(x,t) \in \mathbb U$ is defined by the convolutions
$$
U(x,t)=(f \ast \phi_t)(x), \quad V(x,t)=(f \ast \psi_t)(x).
$$
The parameter $r$ may change when we consider a problem of estimating
the maximal dilatation of the Beurling--Ahlfors extension $F$ in terms of the quasisymmetry 
constant of $f$ related to the doubling constant $\rho$. In particular, when we investigate the asymptotic conformality of 
possible quasiconformal extensions $F(x,t)$ of $f$ as $t \to 0$,
the Beurling--Ahlfors extension of $r=2$ gives a powerful tool, as is shown in Carleson \cite{Ca}. 

Modification and variation to the Beurling--Ahlfors extension have been made by replacing the functions
$\phi$ and $\psi$. These methods are particularly effective for a study of relevant problems in harmonic analysis.
A locally integrable function $h$ on $\mathbb R$ is of {\it BMO} (denoted by $h \in {\rm BMO}(\mathbb R)$) if
$$
\Vert h \Vert_{\rm BMO}=\sup_{I \subset \mathbb R}\frac{1}{|I|} \int_I |h(x)-h_I| dx<\infty,
$$
where the supremum is taken over all bounded intervals $I$ on $\mathbb R$ and $h_I$ denotes the integral mean of $h$
over $I$.
Semmes \cite{Se} took $\phi$ and $\psi$ in $C^\infty(\mathbb R)$ supported on $[-1,1]$ such that
$\phi$ is an even function with $\int_{\mathbb R} \phi(x)dx=1$ and $\psi$ is an odd function with $\int_{\mathbb R} x\psi(x)dx=-1$.
It was proved that if a quasisymmetric homeomorphism $f:\mathbb R \to \mathbb R$ is locally absolutely continuous
and $\log \omega$ for $\omega=f'$ is in ${\rm BMO}(\mathbb R)$ with the norm $\Vert \log \omega \Vert_{\rm BMO}$ small,
then this modified Beurling--Ahlfors extension $F$ is quasiconformal such that $\frac{1}{t}|\mu_F(x,t)|^2dxdt$ is a
Carleson measure on $\mathbb U$, where $\mu_F=\bar \partial F/\partial F$ is the complex dilatation of $F$.
Here, a measure $\lambda(x,t)dx dt$ on $\mathbb U$ is called a {\it Carleson measure} if
$$
\Vert \lambda \Vert_c^{1/2}=\sup_{I \subset \mathbb R} \frac{1}{|I|}\int_0^{|I|}\!\!\int_I \lambda(x,t)dxdt<\infty,
$$
where the supremum is also taken over all bounded intervals $I$. The Carleson norm of $\frac{1}{t}|\mu_F(x,t)|^2dxdt$ is estimated in terms of 
$\Vert \log \omega \Vert_{\rm BMO}$. The arguments rely on the John--Nirenberg inequality for BMO functions, so
the assumption on the smallness of the BMO norm is needed for a single application of the Beurling--Ahlfors extension.

In the paper by Fefferman, Kenig and Pipher \cite{FKP}, a variant of the Beurling--Ahlfors extension was also
utilized, where $\phi(x)=\frac{1}{\sqrt{\pi}}e^{-x^2}$ and $\psi(x)=\phi'(x)=\frac{-2x}{\sqrt{\pi}}e^{-x^2}$.
In this case, the $x$-derivative $U_x(x,\sqrt{t})=(\omega \ast \phi_{\sqrt{t}})(x)$ for example is the solution of the heat equation
having $\omega=f'$ as the initial state, which is represented by the {\it heat kernel} $\phi_{\sqrt t}(x)=\frac{1}{\sqrt{\pi t}}e^{-x^2/t}$.
We see from their arguments that
if $f:\mathbb R \to \mathbb R$ is locally absolutely continuous and the derivative $\omega=f'$ is an ${\rm A}_\infty$-weight
introduced by Muckenhoupt (see \cite{CF}), which implies that $\log \omega \in {\rm BMO}(\mathbb R)$, then this variant of 
the Beurling--Ahlfors extension $F$ is quasiconformal and induces a Carleson measure $\frac{1}{t}|\mu_F(x,t)|^2dxdt$ as before.
No assumption on the BMO norm is necessary.

In this present paper, in view of the importance of the arguments in \cite{FKP}, 
we give a rather detailed proof of the aforementioned results by picking up related parts from the original paper and
complementing necessary arguments between the sentences in it. 
Sections 2 and 3 are devoted to these arrangements of the
theorems in \cite{FKP}. Then in Section 4, we adapt the arguments involving the BMO norm in \cite{Se} 
to the variant of the Beurling--Ahlfors extension $F$ of $f$ given by the heat kernel.
To this end, we generalize the proof in \cite{Se} for $\phi$ and $\psi$ of compact supports
to those rapidly decreasing functions of non-compact supports,
which is a novelty in this paper. As a result, we obtain an estimate of the Carleson norm of
$\frac{1}{t}|\mu_F(x,t)|^2dxdt$ in terms of the BMO norm of $\log \omega$ when it is small. This is valid even if 
the smallness is localized as in the case mentioned next. 

It is said that $h \in {\rm BMO}(\mathbb R)$ is of {\it VMO} if
$$ 
\lim_{|I| \to 0}\frac{1}{|I|} \int_I |h(x)-h_I| dx=0.
$$
Correspondingly, a Carleson measure $\lambda(x,t)dxdt$ is {\it vanishing} if
$$
\lim_{|I| \to 0}\frac{1}{|I|}\int_0^{|I|}\!\!\int_I \lambda(x,t)dxdt=0.
$$
Thus, we can show that if $\log \omega \in {\rm VMO}(\mathbb R)$ for an ${\rm A}_\infty$-weight $\omega=f'$ then 
the variant of the Beurling--Ahlfors extension $F$ of $f$ by the heat kernel yields that 
the Carleson measure $\frac{1}{t}|\mu_F(x,t)|^2dxdt$ is vanishing.
This is a problem asked by Shen \cite{Sh19} in his study of the VMO Teichm\"uller space on the real line.

\section{Heat equation for ${\rm A}_\infty$-weights}\label{sec1}

This section is an exposition of a part of Section 3 of 
Fefferman, Kenig and Pipher \cite{FKP}.

For an ${\rm A}_\infty$-weight $\omega$ on the real line $\mathbb R$, we define
$$
u(x,t)=(\omega \ast \Phi_t)(x) \quad (x \in \mathbb R,\ t>0),
$$
where $\Phi_t(x)$ is the heat kernel given by
$$
\Phi_t(x)=\frac{1}{\sqrt{\pi t}}\, e^{-\frac{x^2}{t}}.
$$
We remark that this $\Phi_t$ comes from
$\Phi(x)=\frac{1}{\sqrt{\pi}} e^{-x^2}$ and
the definition of $\varphi_t(x)=\frac{1}{\sqrt{t}}\varphi(\frac{x}{\sqrt{t}})$ for a general function $\varphi$ in this section is slightly different from that in the other sections.
This satisfies $H \Phi_t(x)=0$ 
for 
$$
H=\frac{\partial}{\partial t}-\frac{\partial^2}{4\partial x^2},
$$
and hence $Hu(x,t)=0$. 

The solution $u$ for the heat equation with the initial state $\omega$ satisfies the following:

\begin{lemma}\label{lem1}
There are constants $c, C>0$ such that
$$
c u(x,t) \leq \frac{1}{\sqrt t}\int_{|x-y|<\sqrt t} \omega(y)dy \leq C u(x,t)
$$
for any $x \in \mathbb R$ and $t>0$.
\end{lemma}

\begin{proof}
We decompose the integral for the convolution as
$$
u(x,t)=\int_{|x-y|<\sqrt t} \omega(y)\Phi_t(x-y)dy+\sum_{n=1}^\infty \int_{2^{n-1} \sqrt{t} \leq |x-y|<2^n \sqrt{t}} \omega(y)\Phi_t(x-y)dy.
$$
Then, the second inequality in the statement is given by
\begin{align*}
u(x,t) &\geq \int_{|x-y|<\sqrt t} \omega(y)\Phi_t(x-y)dy \\
&\geq \Phi_t(\sqrt t) \int_{|x-y|<\sqrt t} \omega(y)dy 
= \frac{1}{e \sqrt \pi} \frac{1}{\sqrt t} \int_{|x-y|<\sqrt t} \omega(y)dy
\end{align*}
with $C=e \sqrt \pi$. 

For the first inequality, we use the doubling property of $\omega$: 
there is a constant $\rho>1$ such that
$$
\int_{2I} \omega(x)dx \leq \rho \int_{I} \omega(x)dx
$$
for any bounded interval $I \subset \mathbb R$. Then, we have
\begin{align*}
\sum_{n=1}^\infty \int_{2^{n-1} \sqrt{t} \leq |x-y|<2^n \sqrt{t}} \omega(y)\Phi_t(x-y)dy
&\leq \sum_{n=1}^\infty \Phi_t(2^{n-1}\sqrt t) \int_{|x-y|<2^n \sqrt{t}} \omega(y)dy\\
&\leq \sum_{n=1}^\infty \frac{\rho^n}{e^{4^{n-1}}\sqrt{\pi t}} \int_{|x-y|< \sqrt{t}} \omega(y)dy.
\end{align*}
Hence, for $c^{-1}=(1+\sum_{n=1}^\infty \frac{\rho^n}{e^{4^{n-1}}})/\sqrt \pi<\infty$,
we obtain the first inequality.
\end{proof}

We consider the spacial derivative $u'(x,t)=\frac{\partial}{\partial x}u(x,t)=(\omega \ast (\Phi_t)')(x)$,
where
$$
(\Phi_t)'(x)=-\frac{2x}{t\sqrt{\pi t}}\, e^{-\frac{x^2}{t}}=-\frac{2x}{t}\Phi_t(x).
$$

\begin{lemma}\label{lem2}
There is a constants $C_1>0$ such that
$$
|u'(x,t)| \leq \frac{C_1}{\sqrt t}\, u(x,t)
$$
for any $x \in \mathbb R$ and $t>0$.
\end{lemma}

\begin{proof}
We also use the decomposition 
$$
u'(x,t)=\int_{|x-y|<\sqrt t} \omega(y)(\Phi_t)'(x-y)dy
+\sum_{n=1}^\infty \int_{2^{n-1} \sqrt{t} \leq |x-y|<2^n \sqrt{t}} \omega(y)(\Phi_t)'(x-y)dy.
$$
The first term is estimated as
\begin{align*}
\left|\int_{|x-y|<\sqrt t} \omega(y)(\Phi_t)'(x-y)dy \right| 
&\leq \int_{|x-y|<\sqrt t} \omega(y)|(\Phi_t)'(x-y)|dy\\
& \leq \frac{\sqrt 2}{\sqrt{e \pi}\,t}  \int_{|x-y|<\sqrt t} \omega(y)dy.
\end{align*}
Here, we used the fact that 
$\max_{|x|<\sqrt t} |(\Phi_t)'(x)|=\sqrt 2/(\sqrt{e \pi}\,t)$ attained at $x=\sqrt{t}/\sqrt{2}$.
The remainder terms are estimated in the same way as before:
\begin{align*}
\sum_{n=1}^\infty \left |\int_{2^{n-1} \sqrt{t} \leq |x-y|<2^n \sqrt{t}} \omega(y)(\Phi_t)'(x-y)dy \right|
&\leq \sum_{n=1}^\infty |(\Phi_t)'(2^{n-1}\sqrt t)| \int_{|x-y|<2^n \sqrt{t}} \omega(y)dy\\
&\leq \sum_{n=1}^\infty \frac{(2\rho)^n}{e^{4^{n-1}}\sqrt{\pi}\, t} \int_{|x-y|< \sqrt{t}} \omega(y)dy.
\end{align*}
Then, by using the constant $C>0$ in Lemma \ref{lem1}, we have
$$
|u'(x,t)| \leq \left( \frac{\sqrt 2}{\sqrt{e \pi}}+\sum_{n=1}^\infty \frac{(2\rho)^n}{e^{4^{n-1}}\sqrt{\pi}} \right) 
\frac{1}{t} \int_{|x-y|< \sqrt{t}} \omega(y)dy
\leq \frac{C_1}{\sqrt t}\, u(x,t)
$$
for the appropriate constant $C_1>0$ involving $C$.
\end{proof}

We prove the following necessary condition for a weight $\omega$ to be in ${\rm A}_{\infty}(\mathbb R)$.
This corresponds to \cite[Theorem 3.4]{FKP}.

\begin{theorem}\label{thm3}
The solution $u$ for an initial state $\omega \in {\rm A}_{\infty}(\mathbb R)$ satisfies that 
$$
\frac{1}{t} \int_0^{t^2} \!\! \int_{|x-x_0|<t} \frac{u'(x,s)^2}{u(x,s)^2} dx ds
$$
is uniformly bounded for any $x_0 \in \mathbb R$ and $t>0$.
\end{theorem}

\begin{proof}
A simple computation using $Hu=0$ shows that
\begin{align*}
H \log u(x,t)&=\frac{\partial}{\partial t} \log u(x,t)-\frac{\partial^2}{4 \partial x^2} \log u(x,t)\\
&=\frac{H u(x,t)}{u(x,t)}+\frac{u'(x,t)^2}{4 u(x,t)^2}=\frac{u'(x,t)^2}{4 u(x,t)^2}.
\end{align*}
This yields that
\begin{align*}
&\quad \frac{1}{4}\int_0^{t^2} \!\! \int_{|x-x_0|<t} \frac{u'(x,s)^2}{u(x,s)^2} dx ds
=\int_0^{t^2} \!\! \int_{|x-x_0|<t} H \log u(x,s) dx ds \\
&=\int_{|x-x_0|<t}\! \int_0^{t^2} \frac{\partial}{\partial s} \log u(x,s) ds dx
-\int_0^{t^2} \!\! \int_{|x-x_0|<t} \frac{\partial^2}{4\partial x^2} \log u(x,s) dx ds \\
&=\int_{|x-x_0|<t} (\log u(x,t^2)-\log \omega(x)) dx
-\frac{1}{4}\int_0^{t^2} \left(\frac{u'(x_0+t,s)}{u(x_0+t,s)}- \frac{u'(x_0-t,s)}{u(x_0-t,s)}\right)ds. \tag{1}
\end{align*}
Here, by Lemma \ref{lem2}, the second term of $(1)$ is bounded by
$$
\frac{2}{4} \int^{t^2}_0 \frac{C_1}{\sqrt s} ds=C_1 t.
$$
For the estimate of the first term of $(1)$, we use the following result.

\begin{lemma}\label{lem4}
Any weight $\omega \in {\rm A}_{\infty}(\mathbb R)$ satisfies that 
$$
\frac{1}{t} \int_{|x-x_0|<t} (\log u(x,t^2)-\log \omega(x)) dx
$$
is uniformly bounded for any $x_0 \in \mathbb R$ and $t>0$.
\end{lemma}

\begin{proof}
Lemma \ref{lem1} implies that 
$$
u(x,t^2) \asymp \frac{1}{t} \int_{|x-y|<t} \omega(y)dy
$$
uniformly for all $x \in \mathbb R$ and $t>0$ (the notation $\asymp$ is used in this sense hereafter). Moreover, the doubling property of $\omega$
implies that if $x$ satisfies $|x-x_0|<t$ for
a fixed $x_0$, then
$$
\int_{|x-y|<t} \omega(y)dy \asymp \int_{|x_0-y|<t} \omega(y)dy.
$$
Hence, if $|x-x_0|<t$ then 
$$
u(x,t^2) \asymp \frac{1}{t} \int_{|x_0-y|<t} \omega(y)dy
$$
independently of $t>0$. This shows that there is a constant $C_2>0$ such that
$$
\int_{|x-x_0|<t} \log u(x,t^2) dx \leq 2t \log \left(\frac{1}{2t} \int_{|x_0-y|<t} \omega(y)dy \right) +tC_2.
$$

It is known that $\omega \in {\rm A}_\infty(\mathbb R)$ if and only if
$$
\frac{1}{|I|} \int_I \omega(x)dx \asymp \exp \left(\frac{1}{|I|} \int_I \log \omega(x) dx \right)
$$
for every bounded interval $I \subset \mathbb R$ (see \cite{Hr}). This implies that there is a constant $C_3>0$ such that
$$
2t\log \left(\frac{1}{2t} \int_{|x_0-y|<t} \omega(y)dy\right)-\int_{|x_0-y|<t} \log \omega(y)dy \leq C_3 t.
$$
Combining this with the previous estimate, we have
$$
\int_{|x-x_0|<t} (\log u(x,t^2)-\log \omega(x)) dx \leq (C_2+C_3)t,
$$
which proves the required inequality.
\end{proof} 

\noindent
{\it Proof of Theorem \ref{thm3} continued.} By the above estimates of the last line of $(1)$, we obtain that
$$
\frac{1}{4}\int_0^{t^2} \!\! \int_{|x-x_0|<t} \frac{u'(x,s)^2}{u(x,s)^2} dx ds \leq (C_1+C_2+C_3)t,
$$
and thus proves the statement.
\end{proof}

We can obtain a similar result to Theorem \ref{thm3} by using 
the spacial second derivative $u''(x,t)=\frac{\partial^2}{\partial x^2}u(x,t)=(\omega \ast (\Phi_t)'')(x)$.
This result is necessary for the proof of Theorem \ref{carleson} in the next section, so 
we formulate this especially and give a proof in our paper.
We apply the argument in \cite[Lemma 3.2]{FKP}.

\begin{theorem}\label{thm5}
The solution $u$ for an initial state $\omega \in {\rm A}_{\infty}(\mathbb R)$ satisfies that
$$
\frac{1}{t} \int_0^{t^2} \!\! \int_{|x-x_0|<t} s\,\frac{u''(x,s)^2}{u(x,s)^2} dx ds
$$
is uniformly bounded for any $x_0 \in \mathbb R$ and $t>0$.
\end{theorem}

\begin{proof}
First, we will find an appropriate function $\eta$ on $\mathbb R$ that satisfies $\Phi'=\eta \ast \Phi_{\frac{1}{2}}$.
We use the Fourier transformation
$$
\mathcal F(h)(\xi)=\frac{1}{\sqrt{2\pi}}\int_{-\infty}^{\infty}h(x)e^{-i\xi x}dx
$$
of a function $h$ on $\mathbb R$. Then, the desired function $\eta$ should satisfy that
$$
i\xi\mathcal F(\Phi)(\xi)= \mathcal F(\Phi')(\xi)=\sqrt{2\pi}\mathcal F(\eta)(\xi) \cdot \mathcal F(\Phi_{\frac{1}{2}})(\xi),
$$
and hence
$$
\mathcal F(\eta)(\xi)=\frac{i}{\sqrt{2\pi}}\xi\, \frac{\mathcal F(\Phi)(\xi)}{\mathcal F(\Phi_{\frac{1}{2}})(\xi)}
=\frac{i}{\sqrt{2\pi}}\xi\, \frac{\frac{1}{\sqrt{2\pi}}e^{-\frac{\xi^2}{4}}}{\frac{1}{\sqrt{2\pi}}e^{-\frac{\xi^2}{8}}}
=\frac{i}{\sqrt{2\pi}}\xi\,e^{-\frac{\xi^2}{8}}.
$$
Therefore, by the inverse Fourier transformation
$$
\mathcal F^{-1}(\hat h)(x)=\frac{1}{\sqrt{2\pi}}\int_{-\infty}^{\infty}\hat h(\xi)e^{i\xi x}d\xi,
$$
we have that
$$
\eta(x)=\mathcal F^{-1}\left(\frac{i}{\sqrt{2\pi}}\xi\,e^{-\frac{\xi^2}{8}}\right)(x)=-\frac{4\sqrt{2}}{\sqrt{\pi}} x e^{-2x^2}.
$$

We represent $u''(x,t)=(\omega \ast (\Phi_t)'')(x)$ by using $\eta_t(x)=\frac{1}{\sqrt{t}}\eta(\frac{x}{\sqrt{t}})$.
We note that
$$
(\Phi_t)'(x)=\frac{1}{\sqrt{t}}(\Phi')_t(x)=\frac{1}{\sqrt{t}}(\eta \ast \Phi_{\frac{1}{2}})_t(x)
=\frac{1}{\sqrt{t}}(\eta_t \ast \Phi_{\frac{t}{2}})(x).
$$
Hence,
$$
u''(x,t)=(\omega \ast (\Phi_t)')'(x)=\frac{1}{\sqrt{t}}(\omega \ast \eta_t \ast (\Phi_{\frac{t}{2}})')(x).
$$
Using this, we obtain that
\begin{align*}
u''(x,t)^2 &=\frac{1}{t}\left(\int_{-\infty}^{\infty} \eta_t(x-y) (\omega \ast (\Phi_{\frac{t}{2}})')(y) dy \right)^2\\
& \leq \frac{1}{t}\left(\int_{-\infty}^{\infty} |\eta_t(x-y)|dy \right)
\left(\int_{-\infty}^{\infty} |\eta_t(x-y)|(\omega \ast (\Phi_{\frac{t}{2}})')(y)^2 dy \right)\\
& = \frac{2}{t}\int_{-\infty}^{\infty} |\eta_t(x-y)|u'(y,t/2)^2 dy,
\end{align*}
where we used $\int |\eta_t(y)|dy=\int |\eta(y)|dy=2\sqrt{2}/\sqrt{\pi}<2$ in the last equation.

To dominate the integrand on question, we use an inequality $u(x,t/2)^2 \leq D_1 u(x,t)^2$ for some constant $D_1>0$,
which is obtained by Lemma \ref{lem1}.
Then,
$$
\frac{u''(x,t)^2}{u(x,t)^2}=\frac{u(x,t/2)^2}{u(x,t)^2}\frac{u''(x,t)^2}{u(x,t/2)^2}
\leq D_1 \frac{u''(x,t)^2}{u(x,t/2)^2}.
$$
Therefore,
\begin{align*}
&\quad \frac{1}{t} \int_0^{t^2} \!\! \int_{|x-x_0|<t} s\,\frac{u''(x,s)^2}{u(x,s)^2} dx ds\\
&\leq \frac{2D_1}{t} \int_0^{t^2} \!\! \int_{|x-x_0|<t}
\left(\int_{-\infty}^{\infty} |\eta_s(x-y)|u'(y,s/2)^2 dy\right) \frac{1}{u(x,s/2)^2} dx ds\\
&=\frac{4D_1}{t} \int_0^{\frac{t^2}{2}} \left(
\iint_{\begin{subarray}{c}|x-x_0|<t \\ -\infty<y<\infty \end{subarray}}|\eta_{2s}(x-y)|\frac{u(y,s)^2}{u(x,s)^2} \frac{u'(y,s)^2}{u(y,s)^2} dxdy \right) ds\\
&=\frac{4D_1}{t} \int_0^{\frac{t^2}{2}} \left( 
\iint_{\begin{subarray}{c}|x-x_0|<t \\ |x-y|<\sqrt{s} \end{subarray}}|\eta_{2s}(x-y)|\frac{u(y,s)^2}{u(x,s)^2} \frac{u'(y,s)^2}{u(y,s)^2} dxdy \right) ds\\
&+\frac{4D_1}{t} \int_0^{\frac{t^2}{2}} \sum_{k=1}^\infty 
\left(\iint_{\begin{subarray}{c}|x-x_0|<t \\ 2^{k-1}\sqrt{s} \leq |x-y|<2^k \sqrt{s} \end{subarray}}
|\eta_{2s}(x-y)|\frac{u(y,s)^2}{u(x,s)^2} \frac{u'(y,s)^2}{u(y,s)^2} dxdy\right) ds. \tag{2}
\end{align*}

We estimate the first term of the last line $(2)$.
If $|x-y|<\sqrt{s}$, then by Lemma \ref{lem1} and the doubling property of $\omega$ with the constant $\rho$, we have
$$
\frac{u(y,s)^2}{u(x,s)^2} \leq D_2 \left(\frac{\int_{|y-z|<\sqrt{s}}\omega(z)dz}{\int_{|x-z|<\sqrt{s}}\omega(z)dz}\right)^2 \leq D_2 \rho^2
$$
for some constant $D_2>0$. Moreover, if $|x-y|<\sqrt{s} \leq t/\sqrt{2}$ and $|x-x_0|<t$, then
$|y-x_0|<2t$. Hence, the integrand $I_0(s)$ by $ds$ is estimated as
\begin{align*}
I_0(s)&= 
\iint_{\begin{subarray}{c}|x-x_0|<t \\ |x-y|<\sqrt{s} \end{subarray}}|\eta_{2s}(x-y)|\frac{u(y,s)^2}{u(x,s)^2} \frac{u'(y,s)^2}{u(y,s)^2} dxdy\\
& \leq \frac{(\max |\eta|)D_2 \rho^2}{\sqrt{2s}}
\iint_{\begin{subarray}{c}|y-x_0|<2t \\ |x-y|<\sqrt{s} \end{subarray}} \frac{u'(y,s)^2}{u(y,s)^2} dxdy\\
& \leq D_2 \rho^2
\int_{|y-x_0|<2t} \frac{u'(y,s)^2}{u(y,s)^2} dy,
\end{align*}
where we used a fact that $\max_{x \in \mathbb R} |\eta(x)|$ is $2\sqrt{2}e^{-\frac{1}{2}}/\sqrt{\pi}<1$ 
attained at $x=\pm 1/2$ in the last inequality.
Then, by letting the uniform bound in Theorem \ref{thm3} $C_0>0$, this theorem shows that
\begin{align*}
\frac{4D_1}{t} \int_0^{\frac{t^2}{2}} I_0(s) ds
&\leq \frac{8D_1D_2 \rho^2}{2t} \int_0^{(2t)^2}\!\!\int_{|y-x_0|<2t} \frac{u'(y,s)^2}{u(y,s)^2} dyds \leq 8C_0D_1D_2 \rho^2.
\end{align*}

Next, we consider the second term of $(2)$.
If $|x-y|<2^k \sqrt{s}$ for $k \geq 1$, then
$$
\frac{u(y,s)^2}{u(x,s)^2} \leq D_2 \rho^{2(k+1)}.
$$
Moreover, if $|x-y|<2^k\sqrt{s} \leq 2^kt/\sqrt{2}$ and $|x-x_0|<t$, then
$|y-x_0|<2^{k+1}t$.
Furthermore,
if $2^{k-1} \sqrt{s} \leq |x-y|$ in addition, then
$$
|\eta_{2s}(x-y)|=\frac{2\sqrt{2}}{\sqrt{\pi}s}|x-y|e^{-\frac{2|x-y|^2}{2s}} < \frac{2^{k+1}}{\sqrt{s}}e^{-4^{k-1}}.
$$

Hence, the integrand $I_k(s)$ $(k \geq 1)$ by $ds$ is estimated as
\begin{align*}
I_k(s)&=
\iint_{\begin{subarray}{c}|x-x_0|<t \\ 2^{k-1}\sqrt{s} \leq |x-y|<2^k \sqrt{s} \end{subarray}}
|\eta_{2s}(x-y)|\frac{u(y,s)^2}{u(x,s)^2} \frac{u'(y,s)^2}{u(y,s)^2} dxdy\\
& \leq D_2 \rho^{2(k+1)} \frac{2^{k+1}}{\sqrt{s}}e^{-4^{k-1}}
\iint_{\begin{subarray}{c}|y-x_0|<2^{k+1}t \\ |x-y|<2^k \sqrt{s} \end{subarray}}
 \frac{u'(y,s)^2}{u(y,s)^2} dxdy\\
 & \leq D_2 (2\rho)^{2(k+1)} e^{-4^{k-1}}
\int_{|y-x_0|<2^{k+1}t}
 \frac{u'(y,s)^2}{u(y,s)^2} dy.
\end{align*}
Then, we have
\begin{align*}
&\quad\frac{4D_1}{t} \int_0^{\frac{t^2}{2}} I_k(s) ds\\
&\leq 4D_1D_2(2\rho)^{2(k+1)} e^{-4^{k-1}} \frac{1}{t}\int_0^{\frac{t^2}{2}}\!\! \int_{|y-x_0|<2^{k+1}t}
 \frac{u'(y,s)^2}{u(y,s)^2} dy ds\\
&\leq D_1D_2(4\rho)^{2(k+1)} e^{-4^{k-1}} \left(\frac{1}{2^{k+1}t}\int_0^{(2^{k+1}t)^2}\!\! \int_{|y-x_0|<2^{k+1}t}
 \frac{u'(y,s)^2}{u(y,s)^2} dy ds\right)\\
&\leq C_0D_1D_2(4\rho)^{2(k+1)} e^{-4^{k-1}},
\end{align*}
where the last inequality is also due to Theorem \ref{thm3}.

Combining these two estimates for $(2)$, we obtain that
$$
\frac{1}{t} \int_0^{t^2} \!\! \int_{|x-x_0|<t} s\,\frac{u''(x,s)^2}{u(x,s)^2} dx ds
\leq C_0D_1D_2 \sum_{k=0}^\infty (4\rho)^{2(k+1)} e^{-4^{k-1}}<\infty.
$$
Thus, we complete the proof of the theorem.
\end{proof}

\medskip
\section{The heat kernel variant of the Beurling--Ahlfors extension}

This section is an exposition of a part of Section 4 of 
Fefferman, Kenig and Pipher \cite{FKP}. 

Let $\phi(x)=\frac{1}{\sqrt \pi}e^{-x^2}$ and $\psi(x)=\phi'(x)=-2x \phi(x)$.
For any $t>0$, we set $\phi_t(x)=\frac{1}{t}\phi(\frac{x}{t})$ and $\psi_t(x)=\frac{1}{t}\psi(\frac{x}{t})$.
For a doubling weight $\omega$ on $\mathbb R$, we define a quasisymmetric homeomorphism 
$f:\mathbb R \to \mathbb R$ by $f(x)=\int^x_0 \omega(y)dy$. Then, we extend $f$ to the upper half-plane 
$\mathbb U=\{(x,t) \mid t>0\}$ by setting a differentiable map $F(x,t)=(U,V)$ for $U(x,t)=(f \ast \phi_t)(x)$ and
$V(x,t)=(f \ast \psi_t)(x)$.

In Section 2, we consider the heat kernel $\Phi_t(x)=\frac{1}{\sqrt{\pi t}}\, e^{-\frac{x^2}{t}}$ and the solution
$u(x,t)=(\omega \ast \Phi_t)(x)$ of the heat equation with the initial state $\omega$. Then, the partial derivatives of
$U$ and $V$ are represented as follows:
\begin{align*}
U_x&=\frac{\partial U}{\partial x}=(\omega \ast \phi_t)(x)=(\omega \ast \Phi_{t^2})(x)=u(x,t^2);\\
V_x&=\frac{\partial V}{\partial x}=(\omega \ast \psi_t)(x)=t (\omega \ast (\Phi_{t^2})')(x)=tu'(x,t^2);\\
U_t&=\frac{\partial U}{\partial t}=(f \ast \frac{\partial \phi_t}{\partial t})(x)
=\frac{1}{2}(\omega \ast \psi_t)(x)=\frac{1}{2}V_x;\\
V_t&=\frac{\partial V}{\partial t}=(f \ast \frac{\partial \psi_t}{\partial t})(x)
=U_x+\frac{t^2}{2}(\omega \ast (\phi_t)'')(x)
=2(\omega \ast \widetilde \phi_t)(x)
\end{align*}
for $\widetilde \phi_t(x)=\frac{x^2}{t^2}\phi_t(x)$.
We note here that
\begin{align*}
\frac{\partial \phi_t}{\partial t}(x)&=\frac{t}{2}\frac{\partial^2 \phi_t}{\partial x^2} =\frac{1}{2}(\psi_t)'(x); \\ 
\frac{\partial \psi_t}{\partial t}(x)&=\frac{\partial \phi_t}{\partial x}+\frac{t^2}{2} \frac{\partial^3 \phi_t}{\partial x^3}
=(\phi_t)'+\frac{t^2}{2}(\phi_t)'''
=2 (\widetilde \phi_t)'(x).
\end{align*}

\begin{proposition}\label{cor}
$|V_x(x,t)|=2|U_t(x,t)| \leq C_1 U_x(x,t)$ whereas $|V_t(x,t)| \asymp U_x(x,t)$.
\end{proposition}

\begin{proof}
Lemma \ref{lem2} implies that 
$$
\frac{|V_x(x,t)|}{U_x(x,t)}=\frac{2|U_t(x,t)|}{U_x(x,t)}=\frac{t |u'(x,t^2)|}{u(x,t^2)} \leq C_1.
$$
Since $V_t(x,t)=2(\omega \ast \widetilde \phi_t)(x)$ and $U_x(x,t)=(\omega \ast \phi_t)(x)$, the latter statement follows from the next lemma.
\end{proof}

\begin{lemma}\label{doubling}
For a doubling weight $\omega$ on $\mathbb R$, we have
$$
\int_{\mathbb R} \omega(y) \phi_t(x-y)dy \asymp 
\frac{1}{t}\int_{\mathbb R} \omega(y) |x-y| \phi_t(x-y)dy \asymp \frac{1}{t^2}\int_{\mathbb R} \omega(y) |x-y|^2 \phi_t(x-y)dy
$$
uniformly for any $x \in \mathbb R$ and $t>0$.
\end{lemma}

\begin{proof}
An inequality
$$
C \int_{\mathbb R} \omega(y) \phi_t(x-y)dy \geq \frac{1}{t} \int_{\mathbb R} \omega(y) |x-y| \phi_t(x-y)dy
$$
for some $C>0$ is essentially given in Lemma \ref{lem2}. By a similar argument, we can also show that
$$
\frac{C}{t} \int_{\mathbb R} \omega(y) |x-y| \phi_t(x-y)dy \geq \frac{1}{t^2} \int_{\mathbb R} \omega(y) |x-y|^2 \phi_t(x-y)dy
$$
for some $C>0$ possibly different. Hence, we have only to prove that
$$
\frac{C}{t^2} \int_{\mathbb R} \omega(y) |x-y|^2 \phi_t(x-y)dy \geq \int_{\mathbb R} \omega(y) \phi_t(x-y)dy.
$$

Trivial estimates show that
\begin{align*}
\int_{\mathbb R} \omega(y) |x-y|^2 \phi_t(x-y)dy &\geq \int_{t \leq |x-y| < 3t} \omega(y) |x-y|^2 \phi_t(x-y)dy\\
& \geq \frac{9t}{e^9} \int_{t \leq |x-y| < 3t} \omega(y) dy.
\end{align*}
By using the doubling constant $\rho>1$ for $\omega$, we have
$$
\rho \int_{t \leq |x-y| < 3t} \omega(y) dy \geq \int_{|x-y| < 4t} \omega(y) dy \geq \int_{|x-y| < t} \omega(y) dy.
$$
Finally, Lemma \ref{lem1} gives
$$
\frac{1}{t} \int_{|x-y| \leq t} \omega(y) dy \geq c \int_{\mathbb R} \omega(y) \phi_t(x-y)dy.
$$
The combination of these three estimates proves the required inequality.
\end{proof}

The heat kernel variant of the Beurling--Ahlfors extension can be stated as follows.
If we start with a given quasisymmetric homeomorphism $f$ of $\mathbb R$, the only requirement for $f$
in this theorem is that $f$ is locally absolutely continuous. This corresponds to \cite[Lemma 4.4]{FKP}

\begin{theorem}\label{variant}
For a doubling weight $\omega$ on $\mathbb R$, the differentiable map $F:\mathbb U \to \mathbb U$ is a
quasiconformal homeomorphism that extends continuously to the quasisymmetric homeomorphism $f$ of $\mathbb R$.
\end{theorem}

\begin{proof}
For the complex dilatation $\mu_F=\bar \partial F/\partial F$, 
we consider
$$
K_F(x,t)=\frac{1+|\mu_F|^2}{1-|\mu_F|^2}=\frac{|\partial F|^2+|\bar \partial F|^2}{|\partial F|^2-|\bar \partial F|^2}
=\frac{U_x^2+U_t^2+V_x^2+V_t^2}{2(U_xV_t-U_tV_x)},
$$
and prove that this is uniformly bounded. Proposition \ref{cor} implies that
$$
U_x^2+U_t^2+V_x^2+V_t^2 \asymp U_x^2.
$$
The Cauchy--Schwarz inequality implies that
\begin{align*}
U_xV_t&=\frac{2}{t^2}\int_{\mathbb R} \omega(y) \phi_t(x-y)dy 
\int_{\mathbb R} \omega(y) (x-y)^2 \phi_t(x-y)dy\\
& \geq \frac{2}{t^2}\left(\int_{\mathbb R} \omega(y)|x-y| \phi_t(x-y)dy \right)^2.
\end{align*}
Then,
\begin{align*}
U_xV_t-U_tV_x \geq 
\frac{2}{t^2}\left(\int_{\mathbb R} \omega(y)|x-y| \phi_t(x-y)dy\right)^2
-\frac{2}{t^2}\left(\int_{\mathbb R} \omega(y)(x-y) \phi_t(x-y)dy\right)^2.
\end{align*}

We set
$$
I_1(x,t)=\int_{x-y \geq 0} \omega(y)(x-y)\phi_t(x-y)dy;\quad
I_2(x,t)=\int_{x-y \leq 0} \omega(y)|x-y|\phi_t(x-y)dy.
$$
Then, by similar arguments to those in Lemmas \ref{lem1}, \ref{lem2} and \ref{doubling} using the doubling property of $\omega$,
we have
\begin{align*}
I_1(x,t) &\asymp \int_{0 \leq x-y < t} \omega(y)(x-y)\phi_t(x-y)dy\\
&\asymp \int_{-t < x-y \leq 0} \omega(y)|x-y|\phi_t(x-y)dy \asymp I_2(x,t).
\end{align*}
We consider
$$
\frac{\left(\int_{\mathbb R} \omega(y)(x-y) \phi_t(x-y)dy\right)^2}
{\left(\int_{\mathbb R} \omega(y)|x-y| \phi_t(x-y)dy\right)^2}
=\frac{(I_1-I_2)^2}{(I_1+I_2)^2}
=\frac{(1-I_2/I_1)^2}{(1+I_2/I_1)^2}.
$$
If $I_2 \leq I_1$, then
the ratio $I_2/I_1 \leq 1$ is bounded away from $0$. 
Similarly, when $I_1 \leq I_2$, we consider $I_1/I_2 \leq 1$ instead. 
Hence, there is some constant $\varepsilon >0$ such that
$$
\left(\int_{\mathbb R} \omega(y)(x-y) \phi_t(x-y)dy\right)^2 \leq
(1-\varepsilon)\left(\int_{\mathbb R} \omega(y)|x-y| \phi_t(x-y)dy\right)^2.
$$

The above inequality implies that
$$
U_xV_t-U_tV_x \geq \frac{2\varepsilon}{t^2}\left(\int_{\mathbb R} \omega(y)|x-y| \phi_t(x-y)dy\right)^2>0.
$$
Then, Lemma \ref{doubling} shows that the middle term of this inequality is greater than
$\varepsilon' U_x^2$ for some constant $\varepsilon'>0$. This concludes that $K_F(x,t)$ is uniformly bounded, and hence
$\Vert \mu_F \Vert_\infty<1$.

By the property of the heat kernel, we see that 
$U(x,t) \to f(x)$ and $V(x,t) \to 0$ as $t \to 0$. This shows that $F$ extends continuously to $f$ on $\mathbb R$.
Moreover, $F(x,t) \to \infty$ as $(x,t) \to \infty$. Since the Jacobian determinant 
$J_F=U_xV_t-U_tV_x$ is positive at every point as we have seen above, $F$ is a local homeomorphism.
Then, a topological argument deduces that $F$ is an orientation-preserving global diffeomorphism of $\mathbb U$ onto itself.
By $\Vert \mu_F \Vert_\infty<1$, we see that $F$ is quasiconformal.
\end{proof}

If we further assume that $\omega$ is an ${\rm A}_\infty$-weight, that is, 
$f$ is a strongly quasisymmetric homeomorphism, then we see that the complex dilatation $\mu_F$
induces a Carleson measure on $\mathbb U$. This corresponds to \cite[Theorem 4.2]{FKP}.

\begin{theorem}\label{carleson}
For an ${\rm A}_\infty$-weight $\omega$ on $\mathbb R$, the complex dilatation $\mu_F$ of
the quasiconformal diffeomorphism $F:\mathbb U \to \mathbb U$ satisfies that
$\frac{1}{t}|\mu_F(x,t)|^2 dxdt$ is a Carleson measure on $\mathbb U$.
\end{theorem}

\begin{proof}
The complex dilatation $\mu_F=F_{\bar z}/F_z$ $(z=x+it)$ satisfies that
\begin{align*}
|\mu_F|^2=\frac{U_x^2+U_t^2+V_x^2+V_t^2-2J_F}{U_x^2+U_t^2+V_x^2+V_t^2+2J_F}
\leq \frac{2U_t^2+2V_x^2+(U_x-V_t)^2}{U_x^2},
\tag{3}
\end{align*}
where $J_F=U_xV_t-U_tV_x$ is the Jacobian determinant of $F$. Here,
$$
\frac{4U_t^2}{U_x^2}=\frac{V_x^2}{U_x^2}=t^2 \frac{u'(x,t^2)^2}{u(x,t^2)^2},
$$
and by the change of the variables $\sigma=s^2$, we have
\begin{align*}
\frac{1}{t} \int_0^{t} \!\! \int_{|x-x_0|<t} \left(s^2 \frac{u'(x,s^2)^2}{u(x,s^2)^2}\right)\frac{dx ds}{s}
= \frac{1}{2t} \int_0^{t^2} \!\! \int_{|x-x_0|<t} \frac{u'(x,\sigma)^2}{u(x,\sigma)^2}dx d\sigma.
\end{align*}
By Theorem \ref{thm3}, this is uniformly bounded.

Moreover, 
$$
\frac{(U_x-V_t)^2}{U_x^2}=t^4 \frac{u''(x,t^2)^2}{u(x,t^2)^2}
$$
and by the change of the variables again, we have
\begin{align*}
\frac{1}{t} \int_0^{t} \!\! \int_{|x-x_0|<t} \left(s^4 \frac{u''(x,s^2)^2}{u(x,s^2)^2}\right)\frac{dx ds}{s}
= \frac{1}{2t} \int_0^{t^2} \!\! \int_{|x-x_0|<t} \sigma\,\frac{u''(x,\sigma)^2}{u(x,\sigma)^2}dx d\sigma.
\end{align*}
By Theorem \ref{thm5}, this is also uniformly bounded.
Combining these two estimates, we see that
$$
\frac{1}{t} \int_0^{t} \!\! \int_{|x-x_0|<t} |\mu_F(x,s)|^2 \frac{dx ds}{s}
$$
is uniformly bounded, which shows that $\frac{1}{t}|\mu_F(x,t)|^2 dxdt$ is a Carleson measure on $\mathbb U$.
\end{proof}

\medskip
\section{The quasiconformal extension of strongly symmetric homeomorphisms and vanishing Carleson measures}

We assume that $\log \omega$ for an ${\rm A}_\infty$-weight $\omega$ is in ${\rm VMO}(\mathbb R)$, that is, 
$f(x)=\int^x_0 \omega(y)dy$ is a strongly symmetric homeomorphism. Then, we prove that the complex dilatation $\mu_F$
of the quasiconformal diffeomorphism $F:\mathbb U \to \mathbb U$ given in Theorems \ref{variant} and \ref{carleson}
induces a vanishing Carleson measure on the upper half-plane $\mathbb U$. An idea of the argument comes from that by Semmes \cite[Proposition 4.2]{Se}.
This answers the question raised by Shen \cite{Sh19}.

\begin{theorem}\label{vanishing}
For an ${\rm A}_\infty$-weight $\omega$ on $\mathbb R$ with $\alpha=\log \omega \in {\rm VMO}(\mathbb R)$, the complex dilatation $\mu_F$ of
the quasiconformal diffeomorphism $F:\mathbb U \to \mathbb U$ satisfies that
$\frac{1}{t}|\mu_F(x,t)|^2 dxdt$ is a vanishing Carleson measure on $\mathbb U$.
\end{theorem}

\begin{proof}
We use inequality $(3)$ to show that
$$
\frac{1}{t} \int_0^{t} \!\! \int_{|x-x_0|<t} |\mu_F(x,s)|^2 \frac{dx ds}{s} \to 0
$$
uniformly as $t \to 0$. Here, we note that $U_x(x,t)=(\omega \ast \phi_t)(x)$ and each of 
$U_t(x,t)$, $V_x(x,t)$, and $(U_x-V_t)(x,t)$ can be represented by $(\omega \ast \gamma_t)(x)$ explicitly
for a certain $\gamma \in C^{\infty}(\mathbb R)$
such that $\int_{\mathbb R} \gamma(x)dx=0$, $|\gamma|$ is an even function, and $\gamma(x)=O(x^2e^{-x^2})$ $(|x| \to \infty)$. 
For instance, 
$V_x(x,t)=(\omega \ast \psi_t)(x)$
for $\psi(x)=-\frac{2}{\sqrt \pi}xe^{-x^2}$.
We set 
$I(x_0,t)=\{ x \mid |x-x_0|<t\}$.
Then, for the statement, it suffices to prove that 
$$
A(x_0,t)=\frac{1}{t}\int_0^{t} \!\! \int_{I(x_0,t)}\frac{(\omega \ast \gamma_s)(x)^2}{(\omega \ast \phi_s)(x)^2} \frac{dx ds}{s} \to 0
$$
uniformly as $t \to 0$. 

Since $\phi(x) \geq 1/(e\sqrt{\pi})$ for $x \in (-1,1)$, we see that $\phi_t(x-y) \geq 1/(te\sqrt{\pi})$ if $|x-y|<t$.
From this, we have
$$
(\omega \ast \phi_t)(x) \geq \frac{1}{te\sqrt{\pi}}\int_{|x-y|<t}\omega(y) dy.
$$
Moreover, the Cauchy--Schwarz inequality implies that
$$
\left(\frac{1}{2t}\int_{|x-y|<t}\omega(y) dy \right) \left( \frac{1}{2t}\int_{|x-y|<t}\omega(y)^{-1} dy \right) \geq
\left(\frac{1}{2t} \int_{|x-y|<t} \omega(y)^{1/2} \omega(y)^{-1/2}dy\right)^2=1.
$$
Therefore,
$$
(\omega \ast \phi_t)(x)^{-2} \leq c\left( \frac{1}{2t}\int_{|x-y|<t}\omega(y)^{-1} dy \right)^2
$$
for $c=e^2\pi/4$. 
Hence, $A(x_0,t)$ is estimated as follows:
\begin{align*}
A(x_0,t)
&\leq \frac{c}{t} \int_{I(x_0,t)} \int_0^{t}\left( \frac{1}{2s}\int_{|x-y|<s}\omega(y)^{-1} dy \right)^2 
(\omega \ast \gamma_s)(x)^2 \frac{1}{s} ds dx\\
&= \frac{c}{t} \int_{I(x_0,t)} \int_0^{t}\left( \frac{1}{2s}\int_{|x-y|<s}\omega(y)^{-1}1_{I(x_0,Nt)}(y)dy \right)^2 
(\omega \ast \gamma_s)(x)^2 \frac{1}{s} ds dx\\
& \leq \frac{c}{t} \int_{I(x_0,t)} \left(\sup_{s>0} \left\{ \frac{1}{2s}\int_{|x-y|<s}\omega(y)^{-1} 1_{I(x_0,Nt)}(y)dy \right\}\right)^2
\left(\int_0^{t} (\omega \ast \gamma_s)(x)^2 \frac{1}{s}ds \right) dx\\
& \leq \frac{c}{t} \left[\int_{I(x_0,t)} \left(\sup_{s>0} \left\{ \frac{1}{2s}\int_{|x-y|<s}\omega(y)^{-1} 1_{I(x_0,Nt)}(y)dy \right\}\right)^4dx\right]^{1/2}\\
& \qquad \qquad \qquad \qquad \qquad \qquad 
\times \left[\int_{I(x_0,t)}\left(\int_0^{t} (\omega\ast \gamma_s)(x)^2 \frac{1}{s}ds\right)^2 dx\right]^{1/2}.
\tag{4}
\end{align*}
We note that if $x \in I(x_0,t)$ and $|x-y|<s \leq t$
then $y \in I(x_0,Nt)$ for any $N \geq 2$. Thus,
for the equality in the middle line, 
we have replaced $\omega(y)^{-1}$ with
$\omega(y)^{-1} 1_{I(x_0,Nt)}(y)$ taking the product of the characteristic function.  

The integrand of the first factor of $(4)$ is the 4th power of the maximal function
$$
M(\omega^{-1}1_{I(x_0,Nt)})(x)=\sup_{s>0} \left\{ \frac{1}{2s}\int_{|x-y|<s}\omega(y)^{-1} 1_{I(x_0,Nt)}(y)dy \right\}.
$$
The strong $L^4$-estimate of the maximal function implies that
\begin{align*}
\int_{I(x_0,t)} M(\omega^{-1}1_{I(x_0,Nt)})(x)^4 dx &\leq \int_{\mathbb R} M(\omega^{-1}1_{I(x_0,Nt)})(x)^4 dx\\
&\leq C' \int_{\mathbb R} \omega(x)^{-4} 1_{I(x_0,Nt)}(x)dx
=C' \int_{I(x_0,Nt)} \omega(x)^{-4} dx
\end{align*}
for some $C'>0$.

We assume hereafter that $\int_{I(x_0,Nt)} \alpha(x)dx=0$
because adding a constant to $\alpha=\log \omega$ corresponds to a dilation of $f$, which does not change the complex dilatation $\mu_F$.
We remark that once $I(x_0,Nt)$ is given, we can assume this only for $I(x_0,Nt)$ throughout the arguments.
When $x_0$, $t$, or $N$ change, we regard that the assumption is renewed accordingly.
We denote the integral mean of $\alpha$ on a bounded interval $I \subset \mathbb R$ by
$\alpha_I=|I|^{-1} \int_I \alpha(x)dx$.
By the John--Nirenberg inequality, we have
\begin{align*}
&\quad \frac{1}{|I(x_0,Nt)|}\int_{I(x_0,Nt)}\omega(x)^{-4} dx\\  
&\leq \frac{1}{|I(x_0,Nt)|}\int_{I(x_0,Nt)}{\rm exp}(4|\alpha(x)-\alpha_{I(x_0,Nt)}|) dx\\
&= \frac{1}{|I(x_0,Nt)|}\int_{I(x_0,Nt)}({\rm exp}(4|\alpha(x)-\alpha_{I(x_0,Nt)}|) -1)dx + 1\\
&= 4\int_{0}^{\infty}\frac{e^{4\lambda}}{|I(x_0,Nt)|} |\{x \in I(x_0,Nt) : |\alpha(x) - \alpha_{I(x_0,Nt)}|>\lambda\}| d\lambda + 1\\
& \leq 4C_1 \int_0^{\infty} e^{4\lambda} {\rm exp}\left(\frac{-C_2\lambda}{\Vert \alpha \Vert_{{\rm BMO}(I(x_0,Nt))}}\right)d\lambda + 1
= \frac{4 C_1 \Vert \alpha \Vert_{{\rm BMO}(I(x_0,Nt))}}{C_2 - 4\Vert \alpha \Vert_{{\rm BMO}(I(x_0,Nt))}} + 1\\
\end{align*}
for some positive constants $C_1$ and $C_2$, where $\Vert \alpha \Vert_{{\rm BMO}(I)}$ denotes the BMO norm of $\alpha$ on a bounded interval $I$.
Thus, for a sufficiently small $t>0$ with $N$ fixed, this is bounded; as a consequence,  
the integral by $dx$ over $I(x_0,t)$ in the first factor of $(4)$ is bounded by $\widetilde{C'}Nt$
for some uniform constant $\widetilde{C'}>0$.

Next, we  consider the integrand of the second factor of $(4)$.
Since $\int_{\mathbb R} \gamma(x)dx=0$, we can replace the convolution $\omega\ast \gamma_s$ with $(\omega-1)\ast \gamma_s$.
For a sufficiently large $N>0$, we decompose this convolution into the integrals on the interval $I(x_0, Nt)$ and
on its complement $I(x_0, Nt)^c=\mathbb R \setminus I(x_0, Nt)$ and estimate the integrand as
\begin{align*}
&\quad \left(\int_0^{t} ((\omega-1) \ast \gamma_s)(x)^2 \frac{1}{s}ds\right)^2\\
&= \left(\int_0^{t} [((\omega-1) 1_{I(x_0, Nt)}\ast \gamma_s)(x)+((\omega-1) 1_{I(x_0, Nt)^c}\ast \gamma_s)(x)]^2 \frac{1}{s}ds\right)^2\\
&\leq 8 \left( \int_0^{t} ((\omega-1) 1_{I(x_0, Nt)}\ast \gamma_s)(x)^2 \frac{1}{s}ds \right)^2\\
&\qquad \qquad \qquad +8 \left(\int_0^{t} ((\omega-1) 1_{I(x_0, Nt)^c}\ast \gamma_s)(x)^2 \frac{1}{s}ds \right)^2. \tag{5}
\end{align*}

Firstly, we consider the integral of the first term of $(5)$ by $dx$ over $I(x_0,t)$. 
We utilize the Littlewood-Paley function defined by the rapidly decreasing function $\gamma$ with $\int_{\mathbb R} \gamma(x)dx=0$:
$$
S_\gamma((\omega-1)1_{I(x_0, Nt)})(x)=\left(\int_0^{\infty} ((\omega-1) 1_{I(x_0, Nt)}\ast \gamma_s)(x)^2 \frac{1}{s}ds \right)^{1/2}.
$$
Then, the strong $L^4$-estimate of the Littlewood-Paley function (see \cite[p.363]{BCP}) implies that
\begin{align*}
&\quad \int_{I(x_0,t)} \left( \int_0^{t} ((\omega-1) 1_{I(x_0, Nt)}\ast \gamma_s)(x)^2 \frac{1}{s}ds \right)^2 dx \\
&\leq
\int_{I(x_0,t)} S_\gamma((\omega-1)1_{I(x_0, Nt)})(x)^4 dx \\
&\leq \int_{\mathbb R} S_\gamma((\omega-1)1_{I(x_0, Nt)})(x)^4 dx\\
&\leq C'' \int_{\mathbb R} (\omega(x)-1)^41_{I(x_0, Nt)}(x)dx=C'' \int_{I(x_0,Nt)} (\omega(x)-1)^{4} dx
\end{align*}
for some $C''>0$. Here,
applying the John--Nirenberg inequality again with the assumption $\int_{I(x_0,Nt)} \alpha(x)dx=0$,
we have  
\begin{align*}
&\quad \frac{1}{|I(x_0,Nt)|}\int_{I(x_0,Nt)}(\omega(x)-1)^{4} dx\\  
&= \frac{1}{|I(x_0,Nt)|}\int_{I(x_0,Nt)}({\rm exp}(|\alpha(x)-\alpha_{I(x_0,Nt)}|-1)^4 dx\\
&= 4\int_{0}^{\infty}\frac{(e^\lambda-1)^3e^{\lambda}}{|I(x_0,Nt)|} |\{x \in I(x_0,Nt) : |\alpha(x) - \alpha_{I(x_0,Nt)}|>\lambda\}| d\lambda\\
& \leq 4C_1 \int_0^{\infty}e^{4\lambda} {\rm exp}\left(\frac{-C_2\lambda}{\Vert \alpha \Vert_{{\rm BMO}(I(x_0,Nt))}}\right)d\lambda
= \frac{4 C_1 \Vert \alpha \Vert_{{\rm BMO}(I(x_0,Nt))}}{C_2 - 4\Vert \alpha \Vert_{{\rm BMO}(I(x_0,Nt))}}. 
\end{align*}
Thus, for a sufficiently small $t>0$ with $N$ fixed,  
the integral of the first term of $(5)$ by $dx$ over $I(x_0,t)$ is bounded by $\widetilde{C''}Nt \Vert \alpha \Vert_{{\rm BMO}(I(x_0,Nt))}$
for some uniform constant $\widetilde{C''}>0$.

Secondly, we consider the integral of the second term of $(5)$ by $dx$ over $I(x_0,t)$. 
For an estimate of the convolution, we use a fact that the weight $\omega+1$ has the doubling property
with some constant $\rho>1$. We note that $|\gamma|$ is an even function. 
Let $n_0=n_0(s,t,N) \in \mathbb N$ satisfy $2^{n_0-1}=(N-1)t/s$ (we may adjust $N$ so that $n_0$ becomes an integer).
Then, for $x \in I(x_0, t)$, we see that
\begin{align*}
|((\omega-1) 1_{I(x_0, Nt)^c}\ast \gamma_s)(x)| &\leq \int_{|y-x_0| \geq Nt} (\omega(y)+1)|\gamma_s(x-y)|dy\\
&\leq \int_{|y-x_0| \geq (N-1)t} (\omega(y)+1)|\gamma_s(y)|dy\\
&=\sum_{n=n_0}^\infty \int_{2^{n-1}s \leq |y-x_0| < 2^ns}  (\omega(y)+1)|\gamma_s(y)|dy\\
& \leq \sum_{n=n_0}^\infty \rho^n \gamma_s(2^{n-1}s) \int_{|y-x_0| < s}(\omega(y)+1)dy.
\end{align*}

Here, by $\gamma(x)=O(x^2e^{-x^2})$ ($|x| \to \infty$), we have
$$
\rho^n \gamma_s(2^{n-1}s) \leq \frac{D_1}{s} \frac{(4\rho)^n}{e^{4^n}}
$$
for some $D_1>0$. For $n \geq n_0(s,t,N)$, we may assume that $(4\rho)^n/e^{2^n} \leq 1$.
By $2^{n_0-1} \geq N-1$, this holds when $N$ is sufficiently large.
Moreover,
$$
\int_{|y-x_0| < s}(\omega(y)+1)dy \leq 2s+\int_{I(x_0,Nt)}\omega(y)dy \leq 2s+D_2Nt
$$
for some $D_2>0$. This estimate of the integral of $\omega$ over $I(x_0,Nt)$ is carried out in a similar way as before by using the
John--Nirenberg inequality when $t$ is sufficiently small with $N$ fixed. 
Therefore, we obtain that if $x \in I(x_0, t)$ then
\begin{align*}
|((\omega-1) 1_{I(x_0, Nt)^c}\ast \gamma_s)(x)|
&\leq \frac{D_1(2s+D_2Nt)}{s}\sum_{n=n_0}^\infty \frac{(4\rho)^n}{e^{2^n}}\frac{1}{e^{2^{n-1}}}\\
&\leq  \frac{DNt}{s}\exp\left(-\frac{Nt}{2s}\right)
\end{align*}
for some uniform constant $D>0$.

We will complete the estimate concerning the second term of $(5)$. By the above inequality, we have
$$
\int_0^t((\omega-1) 1_{I(x_0, Nt)^c}\ast \gamma_s)(x)^2 \frac{1}{s}ds \leq
(DNt)^2 \int_0^t \frac{\exp\left(-\frac{Nt}{s}\right)}{s^3}ds \leq D^2N^2 e^{-N}
$$
for $x \in I(x_0, t)$. For the last inequality, we have used a fact that 
$$
\max_{0 \leq s \leq t} \frac{\exp\left(-\frac{Nt}{s}\right)}{s^3}=\frac{e^{-N}}{t^3}
$$
whenever $N \geq 3$. Hence,
$$
\int_{I(x_0, t)}\left(\int_0^t((\omega-1) 1_{I(x_0, Nt)^c}\ast \gamma_s)(x)^2 \frac{1}{s}ds \right)^2dx
\leq 2D^4 tN^4 e^{-2N}. 
$$

Finally, we substitute what we have obtained into (4) and complete the proof. By replacing $16D^4$ with $\widetilde D$, we conclude that
\begin{align*}
A(x_0,t) &\leq \frac{c}{t}(\widetilde{C'}Nt)^{1/2}(\widetilde{C''}Nt \Vert \alpha \Vert_{{\rm BMO}(I(x_0,Nt))}+\widetilde DtN^4 e^{-2N})^{1/2}\\
&\leq C(N^2\Vert \alpha \Vert_{{\rm BMO}(I(x_0,Nt))}+N^5 e^{-2N})^{1/2},
\end{align*}
where we cleared up the last line by introducing the final constant $C>0$.
Now, for an arbitrary positive $\varepsilon>0$, we choose a sufficiently large $N>0$ that satisfies
$$
N^5 e^{-2N} \leq \frac{\varepsilon^2}{2C^2},
$$
and fix it. Then, for this fixed $N$, we can find some $\delta>0$ such that if $t \leq \delta$ then
$$
N^2\Vert \alpha \Vert_{{\rm BMO}(I(x_0,Nt))} \leq \frac{\varepsilon^2}{2C^2}.
$$
This is because $\alpha \in {\rm VMO}(\mathbb R)$. Thus, if $t \leq \delta$ then $A(x_0,t) \leq \varepsilon$,
independently of $x_0$.
\end{proof}

\begin{remark}
Conversely, a quasiconformal homeomorphism $F:\mathbb U \to \mathbb U$ with $\frac{1}{t}|\mu_F(x,t)|^2 dxdt$ a
vanishing Carleson measure extends continuously to $f:\mathbb R \to \mathbb R$ as a strongly symmetric homeomorphism.
See \cite{Sh19}.
\end{remark}


\begin{thebibliography}{99}

\bibitem{BCP} A. Benedek, A.P. Calder\'on and R. Panzone, 
Convolution operators on Banach space valued functions. 
Proc. Nat. Acad. Sci. U.S.A. 48 (1962), 356--365. 
\bibitem{BA} A. Beurling and L.V. Ahlfors, The boundary correspondence under quasiconformal mappings, Acta Math. 96 (1956), 125--142.
\bibitem{Ca} L. Carleson, On mappings, conformal at the boundary, J. Anal. Math. 19 (1967), 1--13.
\bibitem{CF} R.R. Coifman and C. Fefferman, Weighted norm inequalities for maximal functions and singular integrals, Studia Math. 51 (1974), 241--250.
\bibitem{FKP} R.A. Fefferman, C.E. Kenig and J. Pipher, The theory of weights and the Dirichlet problems for elliptic equations, Ann. of
Math. 134 (1991), 65--124.
\bibitem{Hr} S.V. Hru\v{s}\v{c}ev, A description of weights satisfying the {$A_{\infty }$} condition of Muckenhoupt,
Proc. Amer. Math. Soc. 90 (1984), 253--257.
\bibitem{Se} S. Semmes, Quasiconformal mappings and chord-arc curves, Trans. Amer. Math. Soc. 306 (1988), 233--263.
\bibitem{Sh19} Y. Shen, VMO-Teichm\"uller space on the real line, preprint.


\end{thebibliography}
\end{document}